\documentclass[eqnumis, pdf]{hjms}
\usepackage{amsfonts}
\usepackage{amsmath}
\usepackage{footnote}
\usepackage{multicol}
\usepackage{booktabs}
\usepackage[flushleft]{threeparttable}
\usepackage[font=small,labelfont=bf]{caption}

\begin{document}
%please do not change this part%%%%%
\hinfo{XX}{x}{2019}{1}{\lpage}{10.15672/HJMS.xx}
%%%%%%%%%%%%%%%%%%%%%%%%%%%%%%
%\giris
%{ } %yazarlar
%{ } %başlık
%{ } %ilk sayfa no
%author, title, first page number
%%%%%%%%%%%%%%%%%%%%%%%%%%%%%%
%Corresponding Author Email:
%
%%%%%%%%%%%%%%%%%%%%%%%%%%%%%%

\markboth{A. Babaee, B. Mashayekhy, H. Mirebrahimi, H. Torabi, M. Abdullahi Rashid, S.Z. Pashaei}{On Topological Homotopy Groups and Relation to Hawaiian Groups}

\title{On Topological Homotopy Groups and Relation to Hawaiian Groups}

\author{Ameneh Babaee,  Behrooz Mashayekhy, Hanieh Mirebrahimi\coraut, Hamid Torabi, Mahdi Abdullahi Rashid, Seyyed Zeynal Pashaei}

\address{Department of Pure Mathematics, Center of Excellence in Analysis on Algebraic Structures, Ferdowsi University of
Mashhad,\\ P.O.Box 1159-91775, Mashhad, Iran.}
\emails{am.babaee@mail.um.ac.ir (A. Babaee), bmashf@um.ac.ir (B. Mashayekhy), (h\_mirebrahimi@um.ac.ir) H. Mirebrahimi, h.torabi@um.ac.ir (H. Torabi), mbinev@mail.um.ac.ir (M. Abdullahi Rashid), Pashaei.seyyedzeynal@stu.um.ac.ir (S.Z. Pashaei)}
\maketitle
\begin{abstract}
By generalizing the whisker topology on the $n$th homotopy group of pointed space $(X, x_0)$, denoted by $\pi_n^{wh}(X, x_0)$, we show that $\pi_n^{wh}(X, x_0)$ is a topological group if $n \ge 2$. Also, we present some necessary and sufficient conditions for $\pi_n^{wh}(X,x_0)$ to be discrete, Hausdorff and indiscrete. Then we prove that $L_n(X,x_0)$ the natural epimorphic image of the Hawaiian group $\mathcal{H}_n(X, x_0)$ is equal to the set of all classes of convergent sequences to the identity in $\pi_n^{wh}(X, x_0)$. As a consequence, we show that $L_n(X, x_0) \cong L_n(Y, y_0)$ if $\pi_n^{wh}(X, x_0) \cong \pi_n^{wh}(Y, y_0)$, but the converse does not hold in general, except for some conditions. Also, we show that on some classes of spaces such as semilocally $n$-simply connected spaces and $n$-Hawaiian like spaces, the whisker topology and the topology induced by the compact-open topology of $n$-loop space coincide. Finally, we show that $n$-SLT paths can transfer $\pi_n^{wh}$ and hence $L_n$ isomorphically along its points.
\end{abstract}% the abstract
  % AMS subject classifications
        % Keywords

%please do not change this part%%%%%%%%%%
\hinfoo{DD.MM.YYYY}{DD.MM.YYYY} %receive, accept
%%%%%%%%%%%%%%%%%%%%%%%%%%%%%%%%%%%

%Accepted Date:
%Received Date:
%DOI:

\section{Introduction and Motivation}\label{sec1}
E.H. Spanier introduced a topology on the fundamental group \cite[Theorem 13]{Spa}, named whisker topology by N. Brodskiy et al. \cite{BroDyd}. It is originally defined on a quotient of the path space introduced in \cite[Definition 4.2]{BroDyd} including the fundamental group as a fibre. It was shown that for a pointed space $(X, x_0)$ the restriction of the whisker topology on $\pi_1(X, x_0)$ is generated by the basis $\bigcup_{[\alpha] \in \pi_1(X,x_0)}$ $\{[\alpha] \pi_1(i) \pi_1(U,x_0)$ $|\ U$  is an open neighbourhood of $x_0$ and $i:U \to X$ is the inclusion map$\}$.

Another topology on the fundamental group was defined in \cite{BroDyd2}, called lasso topology. In general, the lasso topology makes the fundamental group a topological group, but not the whisker topology. As an example, if $\mathbb{HE}^1$ denotes the $1$-dimensional Hawaiian earring, the inverse operation of $\pi_1^{wh}(\mathbb{HE}^1, \theta)$ is not continuous \cite{BroDyd2}. Also, if $\pi_1^{qtop}$ denotes the fundamental group equipped with the topology induced by compact-open topology of $1$-loop space, then the multiplication of $\pi_1^{qtop} ( \mathbb{HE}^1)$  is not continuous \cite{Fab}.
This topology was generalized to higher dimension by F.H. Ghane et al. \cite{Ghane} induced by the compact-open topology of $n$-loop space.

In Section 2, we generalize the whisker topology to the $n$th homotopy group, $n \in \mathbb{N}$, denoted by $\pi_n^{wh}(X, x_0)$, using subgroup topology which makes $\pi_n(X, x_0)$ a left topological group for any pointed space $(X, x_0)$. We show that for $n \ge 2$, the whisker topology  makes  $\pi_n(X, x_0)$ a topological group.

In Section 3, we establish some necessary and sufficient conditions  for  $\pi_n^{wh}(X, x_0)$ to be discrete, Hausdorff, and indiscrete. For instance, an equivalent condition for $\pi_n^{wh} (X, x_0)$ to be discrete, is semi-locally $n$-simply connectedness at $x_0$. Also, we show that any subgroup $H \leqslant \pi_n^{wh}(X, x_0)$ is closed if and only if $X$ is $n$-homotopically Hausdorff relative to $H$ at $x_0$.

It is well-known that a path induces an isomorphism on homotopy groups at its beginning and end points. But this isomorphism is not necessarily continuous. Brodskiy et al. \cite[Proposition 4.10]{BroDyd} showed that the $1$--dimensional Hawaiian earring is a path connected space with non-homeomorphic fundamental groups equipped with the whisker topology at some different points. Moreover, they defined a kind of path, called an SLT-path, which makes the induced isomorphism on fundamental groups continuous.
We generalize SLT-paths to $n$-SLT paths in order to induce continuous isomorphism on the $n$th homotopy groups.

Section 4 discusses the relation between topological homotopy groups and Hawaiian groups.
For $n \ge 1$, the $n$th Hawaiian group was defined as a functor from $hTop_*$, the pointed homotopy category, to $Groups$, the category of groups (see \cite{KarRep}). Assume that $\mathbb{HE}^n = \bigcup_{k \in \mathbb{N}} \mathbb{S}_k^n$ denotes the $n$-dimensional Hawaiian earring  introduced in \cite{EdaKaw}, where $\mathbb{S}_k^n$ is the $n$-sphere  with radius $1/k$ centered at $(1/k, 0, \ldots, 0)$ in $\mathbb{R}^{n +1}$, and $\theta$ denotes the origin.

\begin{definition}[\cite{KarRep}]
Let $(X, x_0)$ be a pointed space, and let $[\cdot]$ denote the class of pointed homotopy. The $n$th Hawaiian group of $(X, x_0)$, is defined by $\mathcal{H}_n(X, x_0)= $ $\{[f]:\  f:(\mathbb{HE}^n , \theta) \to (X, x_0)\}$.
For any $[f], [g] \in \mathcal{H}_n(X, x_0)$, multiplication is induced by $(f*g)|_{\mathbb{S}_k^n} = f|_{\mathbb{S}_k^n} * g|_{\mathbb{S}_k^n}$ ($k \in \mathbb{N}$).
\end{definition}

The operation of the $n$th Hawaiian group implies that for all $n \in \mathbb{N}$, the following map

\begin{equation}
{\varphi}:{\mathcal{H}}_n(X,x_0) \to    \prod_{\aleph_0}{\pi}_n(X,x_0), \tag{I}
\end{equation}
defined by $\varphi([f])=([f{\mid}_{\mathbb{S}_1^n}], [f{\mid}_{\mathbb{S}_2^n}],... )$ is a homomorphism.
For every pointed space $(X, x_0)$, homomorphic image $im( \varphi)$ denoted by $L_n(X,x_0)$ which is equal to a special subset of $ \prod_{\aleph_0}{\pi}_n(X,x_0)$ \cite[Definition 2.6]{1} as follows.

\begin{definition}[\cite{1}]
Let $(X,x_0)$ be a pointed space and $n\geq 1$. Then $L_n(X,x_0)$ is the subset  of ${\prod}_{\aleph_0}\pi_n(X,x_0)$ consisting of all sequences of homotopy classes $\{[f_k]\}$, whenever $im(f_k)\subseteq U$ for all $k\in \mathbb{N}$ except a finite number, if $U$ is an open set containing $x_0$.
\end{definition}
For instance, if $X$ is a metric space, then $L_n(X, x_0)$ is the subset of $\prod_{\aleph_0} \pi_n(X, x_0)$ consisting of all classes of uniform convergent sequences to the constant map at $x_0$.

It was proved that $L_n(X,x_0) = \varphi ( \mathcal{H}_n(X, x_0))$, and hence it is a subgroup of ${\prod}_{\aleph_0}\pi_n(X,x_0)$ (see \cite[Theorem 2.7]{1}). Therefore, one can consider the homomorphism $\varphi$ as an epimorphism from $\mathcal{H}_n(X, x_0)$ onto $L_n(X, x_0)$.

In Section 4, we attend the relation of $L_n(X, x_0)$ and $\pi_n^{wh}(X, x_0)$, for any pointed space $(X, x_0)$, and we see that they are closely dependent on each other. In fact, it is shown that $L_n(X, x_0)$ is equal to the set of all convergent sequences to the identity in $\pi_n^{wh}(X,x_0)$. As a consequence, we see that on $n$-Hawaiian like spaces, the two topologies of $\pi_n^{wh}$ and $\pi_n^{qtop}$ coincide. Then, we prove that $L_n(X, x_0) \cong L_n(Y, y_0)$, whenever $\pi_n^{wh}(X, x_0) \cong \pi_n^{wh}(Y, y_0)$ as left topological groups, for any pointed spaces $(X, x_0)$ and $(Y, y_0)$. It implies a sufficient condition to fix the structure of $L_n$ at different points which is the existence of some two sided small $n$-loop transfer ($n$-SLT) path. Finally, we study two groups $L_1 (\mathbb{HA}, a)$ and $L_1( \mathbb{HA}, \theta)$, where $\mathbb{HA}$ is the harmonic archipelago, $\theta$ is the origin, and $a$ is another point. We prove that $L_1 (\mathbb{HA}, a) \not \cong L_1( \mathbb{HA}, \theta)$ to see that the existence of $n$-SLT paths is necessary to induce isomorphism on $L_n$ and topological homotopy groups at different points.

Throughout this article all homotopies are relative to the base point.

\section{Whisker Topology on Homotopy Groups}\label{sec2}
In this section, we intend to introduce the whisker topology on the $n$th homotopy groups. The whisker topology on the fundamental group has been introduced and discussed by Brodskiy et al. in \cite{BroDyd}.

Let $(X, x_0)$ be a pointed space, and let $n \ge 1$. For each open neighbourhood $U$ of $x_0$ in $X$, the inclusion map $i:U \to X$ induces the natural homomorphism $\pi_n(i) : \pi_n(U,x_0) \rightarrow \pi_n(X,x_0)$. Hence, $\pi_n(i) \big(\pi_n(U, x_0)\big)$ is a subgroup of $\pi_n(X,x_0)$. Also, for any open neighbourhoods $U$ and $V$ containing $x_0$, we have
\begin{equation}\label{eq2.1}
\pi_n(i_1) \big(\pi_n(U \cap V,x_0)\big) \leq \pi_n(i_2) \big(\pi_n(U,x_0)\big) \cap \pi_n(i_3)\big(\pi_n(V, x_0) \big),
\end{equation}
where $i_1$, $i_2$, and $i_3$ are corresponding inclusion maps.
Therefore, the collection of all such subgroups forms a \textit{neighbourhood family} on $\pi_n(X,x_0)$ which is defined as follows.

\begin{definition}[\cite{BogSie}]\label{de2.1}
Let $G$ be a group with the identity element $e$. A nonempty family $\Sigma$ of subgroups of $G$ is called a neighbourhood family whenever for any $S, S' \in \Sigma$, there exists $S'' \in \Sigma$, such that $S'' \subseteq S \cap S'$. Let $g \in G$ and $\Sigma$ be a neighbourhood family, then the set of all cosets $\{gS:\ S \in \Sigma\}$ forms a local basis at $g$. Thus, the set $\{gS:\ g \in G, S \in \Sigma\}$ is a basis for a topology on $G$ which is called a subgroup topology.
The intersection $S_{\Sigma} = \bigcap_{S\in \Sigma} S$ is called the infinitesimal subgroup for the neighbourhood family $\Sigma$.
\end{definition}
Using the above definition, we are going to endow the $n$th homotopy group with a topology called whisker topology. The whisker topology on the fundamental group has been defined as a subspace of a path space introduced in \cite{BroDyd}. Note that one can consider the fundamental group as the $1$st homotopy group.

\begin{definition}
Let $(X, x_0)$ be a pointed space, and $n \ge 1$. By Inequality \eqref{eq2.1},
\[
\Sigma = \{ \pi_n(i) \pi_n(U,x_0) \ |  \mathrm{\ U\ is\ an\  open\  subset\ of\ } X  \mathrm{\ containing\ } x_0 \},
\]
is a neighbourhood family on $\pi_n(X, x_0)$.
The whisker topology on the $n$th homotopy group, $\pi_n(X,x_0)$, of a pointed topological space $(X,x_0) $ is the subgroup topology determined by the neighbourhood family $\Sigma$ which is denoted by $\pi^{wh}_n(X,x_0)$.
\end{definition}
Note that for any $n$-loop $\alpha$ the collection $\Sigma_{[\alpha]} = \{ [\alpha] \pi_n(i) \pi_n(U,x_0)  |\  U$ is an open subset of $X$  containing $x_0\}$ is a local basis at $[\alpha] \in \pi_n^{wh}(X, x_0)$. Then we have the following result.

\begin{lemma}\label{le2.3}
Let $(X, x_0)$ be a pointed space, and let $n \ge 1$. If $X$ has a countable local basis at $x_0$, then $\pi_n^{wh}(X, x_0)$ is first countable.
\end{lemma}

Let $n \ge 1$. Recall that an $n$-loop $\alpha: (\mathbb{S}^n, 1) \to (X, x_0)$ is said to be small if it has a homotopic equivalent in every open neighbourhood of $x_0$ \cite{PasGha}, and  $\pi_n^s(X, x_0)$ denotes the collection of all classes of small $n$-loops at $ x_0$. Let $[\alpha] \in \bigcap \{\pi_n(i) \pi_n(U, x_0)\ | \ U$ is an open neighbourhood of $x_0 \}$, then $\alpha$ has a homotopic representative in any open neighbourhood of $x_0$, that is, $\alpha $ is a small $n$-loop at $ x_0$. Thus, the infinitesimal subgroup of $\pi_n^{wh}(X, x_0)$ is equal to $\pi_n^s(X, x_0)$.
 It is easy to see that $\pi_n^{wh}(X, x_0)$ is indiscrete if and only if $\pi_n^s(X, x_0) = \pi_n(X, x_0)$. As an example, if $\mathbb{HA}$ denotes the harmonic archipelago space, and $\theta$ denotes the origin, then $\pi_1^{wh}(\mathbb{HA}, \theta)$ is indiscrete. Moreover, if $\pi^{wh}_n(X,x_0)$ is discrete, then $\pi^{s}_n(X,x_0)$ is the trivial subgroup. The converse does not hold, in general. As a counterexample, for the $n$-dimensional Hawaiian earring, $\mathbb{HE}^n$  at the origin $\theta$, $\pi^{s}_n(\mathbb{HE}^n,\theta)$ is trivial, but $\pi^{wh}_n(\mathbb{HE}^n, \theta)$ is not discrete (see Example \ref{ex2.4}).

\begin{remark}[\cite{BogSie}]\label{re2.4}
With the previous assumption and notation, for $g \in G$ and $S \in \Sigma$, a basic set $gS$  is both open and closed in the subgroup topology, since the cosets of a given subgroup form a partition of $G$. The subgroup topology is a homogeneous space, since left translations by elements of $G$ determine self-homeomorphisms on $G$. However, the group $G$ is not necessarily a topological group, since the right translation by a fixed element of $G$ is not continuous, in general.
The infinitesimal subgroup is a closed subgroup in the subgroup topology on $G$ induced by $\Sigma$. Indeed, $S_{\Sigma}$ is the closure of the identity $e \in G$, and its coset $gS_{\Sigma}$ is the closure of the element $g \in G$.
\end{remark}
Note that $\pi^{s}_n(X,x_0)$ is a closed subgroup of $\pi^{wh}_n(X,x_0)$, but it may not be open, in general. However, some nice properties occur if it is open. The following proposition generalizes Proposition 2.4 in \cite{paper4}, by a similar argument, for the whisker topology on the $n$th homotopy group ($n \ge 1$).

\begin{proposition}
Let $(X,x_0) $ be a pointed topological space, then the following statements are equivalent.

\begin{enumerate}
\item
$\pi^{s}_n(X,x_0)$ is an open subgroup of $\pi^{wh}_n(X,x_0)$.
\item
Every closed subgroup of $\pi^{wh}_n(X,x_0)$ is an open subgroup.
\item
A subgroup $H$ of $\pi^{wh}_n(X,x_0)$ is open if and only if it is closed.
\item
A subgroup $H$ of $\pi^{wh}_n(X,x_0) $ is open if and only if $\pi^{s}_n(X,x_0) \leq H$.
\end{enumerate}
\end{proposition}

By Remark \ref{re2.4}, every subgroup topology on a given group makes it a homogeneous space, and hence, it is a left topological group. It was shown in \cite[Proposition 2.1] {paper4} that if a subgroup topology on a group makes it a right topological group, then it is a topological group. Since $ \pi_n(X,x_0) $ is abelian, for $ n \geq 2 $, the right translation map $r_{\alpha} : \pi_n(X,x_0) \rightarrow \pi_n(X,x_0) $ by the rule $ r_{\alpha}([\beta])= [\alpha*\beta] $ is equal to the left translation map $ l_{\alpha} : \pi_n(X,x_0) \rightarrow \pi_n(X,x_0) $ by the rule $ l_{\alpha}([\beta]) = [\beta*\alpha] $, for any $ [\alpha] \in \pi_n(X,x_0) $. Hence, $r_{\alpha}$ is continuous for all $[\alpha] \in \pi_n(X, x_0)$. Therefore, $\pi_n^{wh}(X, x_0)$ is a right topological group, too, for $n \ge 2$. As a consequence we have the following result.

\begin{proposition}\label{pr2.4}
Let $ (X,x_0)$ be a pointed space. If $n \ge 2$, then  $ \pi^{wh}_n(X,x_0) $ is a topological group.
\end{proposition}

Since $\pi_1(X, x_0)$ is not necessarily an abelian group, Proposition \ref{pr2.4} does not hold in the case of $ n = 1 $. As an example $ \pi^{wh}_1(\mathbb{HE}^1, \theta) $ is not a topological group \cite{BroDyd2}. For $n =1$, there exists a necessary and sufficient condition called SLTL, established in \cite[Proposition 2]{paper3} for $\pi_1^{wh}(X, x_0)$ to be a topological group.

Fisher et al. \cite[Theorem 4.10 (d)]{FisZas} proved that if $X$ is metric, then so is the path space $\widetilde{X}$, whenever $X$ is shape injective. Also, by Lemma \ref{le2.3}, if $X$ has a countable local basis at $x_0$, then $\pi_n^{wh}(X, x_0)$ is first countable. In the following, we see that for $n \ge 2$, there is sufficient conditions for $\pi_n^{wh}(X, x_0)$ to be metric.

G.R. Conner et al. \cite{ConLam} defined the homotopically Hausdorff property. This property has been extended to $n$-homotopically Hausdorff property by H. Passandideh et al. \cite[Definition 3.3]{PasGha} for $n \ge 1$. A space $X$ is called $n$-homotopically Hausdorff at $x_0$ whenever for each essential $n$-loop $\alpha$ in $X$ at $x_0$, there exists an open neighbourhood $U$ of $x_0$, containing no $n$-loop at $x_0$ homotopic to $\alpha$, that is $\pi_n^s (X, x_0) = \langle e \rangle$.

\begin{corollary}\label{co2.11}
Let $X$ be a space having a countable local basis at $x_0$, and let $n \ge 2$. If $X$ is $n$-homotopically Hausdorff at $x_0$, then $\pi_n^{wh} (X, x_0)$ is a metric topological group.
\end{corollary}

\begin{proof}
By Proposition \ref{pr2.4}, $\pi_n^{wh} (X, x_0)$ is a topological group.  If $X$ is $n$-homotopically Hausdorff at $x_0$, then by \cite[Theorem 2.9 (c)]{BogSie}, $\pi_n^{wh}(X, x_0)$ is Hausdorff and thus, satisfies $T_1$-separation axiom.
Hence, by \cite[Theorem 3.3.12, p. 155]{ArhTka}, $\pi_n^{wh}(X, x_0)$ is metric if and only if it is first countable. Since $X$ has a countable local basis at $x_0$, by Lemma \ref{le2.3},  $\pi_n^{wh}(X, x_0)$ is first countable. Therefore, $\pi_n^{wh}(X, x_0)$ is a metric topological group.
\end{proof}

Note that Ghane et al. in \cite[Page 264]{Ghane} by a filter base which forms a fundamental system of neighborhoods of the identity element gave a topology to the homotopy group $\pi_n(X, x_*)$ denoted by $\pi_n^{lim}(X, x_*)$. It should be mentioned that one can prove the topology of $\pi_n^{lim}(X, x_*)$ coincides with the whisker topology $\pi_n^{wh}(X, x_*)$.

\section{Whisker Topology and Local Properties}\label{sec3}

In this section, we are going to find some relationships between topological properties of $\pi_n^{wh}(X, x_0)$ and local properties of the space $X$ at the base point $x_0$.
Moreover, we discuss  conditions for $\pi_n^{wh}(X, x_0)$ to be invariant with respect to the base point $x_0$.

The following proposition states the equivalence condition for $\pi^{wh}_n(X,x_0)$ to be discrete. Recall from \cite[Definition 3.1]{Ghane} that a pointed topological space $(X,x_0)$ is called  semilocally $n$-simply connected at $x_0$ if there exists an open neighbourhood $U$ at $x_0$ for which any $n$-loop in $U$ based at $x_0$ is nulhomotopic in $X$.
\begin{proposition}\label{pro2.3}
Let $(X,x_0)$ be a pointed space, and let $n \ge 1$. Then $\pi^{wh}_n(X,x_0)$ is discrete if and only if $ X $ is semilocally $n$-simply connected at $x_0$.
\end{proposition}

\begin{proof}
If $ X $ is  semilocally $n$-simply connected at $ x_0 $, then there is an open neighbourhood $ U $ of $ x_0 $ such that $\pi_n(i)\pi_n(U,x_0) $ is trivial. Since $\pi_n(i)\pi_n(U,x_0) \in \Sigma $, then $ \pi^{wh}_n(X,x_0) $ is discrete. Conversely, if $ \pi^{wh}_n(X,x_0) $ is discrete, then the trivial subgroup is open in $ \pi^{wh}_n(X,x_0) $. Since $\Sigma$ is a local basis, there is an open neighbourhood $ U $ of $ x_0 $, such that $\pi_n(i)\pi_n(U,x_0) \subseteq \{e\}$, that is $\pi_n(i)\pi_n(U,x_0) = \{e\}$. Hence $X$ is semi-locally $n$-simply connected at $x_0$.
\end{proof}

H. Fischer et al. \cite[]{FisZas} defined homotopically Hausdorff property relative to $H$, where $H$ is a subgroup of $\pi_1(X , x_0)$. Brodskiy et al. \cite[Definition 4.11]{BroDyd2}  generalized this concept to $(G,H)$-homotopically Hausdorff property, where  $H \leqslant G \leqslant \pi_1(X , x_0)$. A space $X$ is called $(G, H)$-homotopically Hausdorff, if for any $g \in G - H$ and any path $\alpha$ originating at $x_0$, there is an open neighbourhood $U$ of $\alpha(1)$ in $X$ such that none of the elements of $Hg$ can be expressed as $[\alpha * \gamma * \alpha^{-1}]$ for any loop $\gamma$ in $(U, \alpha(1))$.
In the following, we define $n$-homotopically Hausdorff property relative to a pair of subgroups $(G, H)$ at the base point $x_0$, where $H \leqslant G \leqslant \pi_n(X, x_0)$ ($n \ge 1$).

\begin{definition}\label{de2.7}
Let $H \leqslant G \leqslant \pi_n(X , x_0)$, and let $n \ge 1$. We say that $X$ is $n$-homotopically Hausdorff relative to $(G,H)$ at $x_0$, if for each $g \in G - H$, there exists an open neighbourhood $U$ of $x_0$, such that no element of $Hg$ can be expressed as $[\gamma]$, for any $n$-loop $\gamma$ in $(U, x_0)$.
\end{definition}
Note that $X$ is $n$-homotopically Hausdorff relative to $G$, if $X$ is $n$-homotopically Hausdorff relative to $(G, \{e\})$ at $x_0$.
Although, $n$-homotopically Hausdorff property relative to $(G, H)$ at $x_0$ is defined closely to $(G, H)$- homotopically Hausdorff property \cite[Definition 4.11]{BroDyd2}, if $X$ is $1$-homotopically Hausdorff relative to $(G, H)$, $H \leqslant G \leqslant \pi_1(X, x_0)$ at any point in the sense of Definition \ref{de2.7}, it does not need to be $(G, H)$-homotopically Hausdorff in the sense of \cite{BroDyd2}.

It is proved that $X$ is homotopically Hausdorff relative to $(G, H)$, if $H$ is closed in $G$ endowed with a new topology \cite[Lemma 4.14]{BroDyd2} called lasso topology in \cite{BroDyd}. Also, Fisher et al.  \cite[]{FisZas} proved that homotopically Hausdorff relative to a subgroup $H$ is equivalent to Hausdorffness of a path space equipped with a suitable topology. The following theorem presents a similar explanation of \cite[Proposition 4.12 and Lemma 4.16]{BroDyd2}, \cite[Lemma 2.10 and Proposition 6.3]{FisZas}.

\begin{theorem}\label{th2.8}
Let $(X, x_0)$ be a pointed space, $H \leqslant G \leqslant \pi_n^{wh}(X, x_0)$, and $n \ge 1$. Then the following statements are equivalent.\\
%\begin{enumerate}
%\item
$(i)$ $X$ is $n$-homotopically Hausdorff relative to $(G, H)$ at $x_0$.\\
%\item
$(ii)$ $H$ is a closed subgroup of $G$.\\
%\item
$(iii)$ The coset space $\frac{G}{H}$, with the quotient topology, is a homogenous Hausdorff space.
%\end{enumerate}
\end{theorem}

\begin{proof}
\begin{enumerate}
\item
($(i) \Rightarrow (ii)$) Let $X$ be $n$-homotopically Hausdorff relative to $(G, H)$ at $x_0$. Then, for every $g \in G - H$, there exists an open neighbourhood $U_g$ of $x_0$, such that $\pi_n(i) \pi_n(U_g, x_0) \cap Hg = \emptyset$, where $i:U \hookrightarrow X$ is the inclusion map. Assume that $g \in G - H$ and $g \in \overline{H}$. Thus, for each open neighbourhood $V$ of $g$ in $G$, $V \cap H \neq \emptyset$. Put $V = g \pi_n(i) \pi_n(U_g, x_0) \cap G$. Then $(g \pi_n(i) \pi_n(U_g, x_0) \cap G) \cap H  \neq \emptyset$. Since $H \leqslant G$, $g \pi_n(i) \pi_n(U_g, x_0) \cap G \cap H = g \pi_n(i) \pi_n(U_g, x_0) \cap H$. Let $h \in g \pi_n(i) \pi_n(U_g, x_0) \cap H$. Then $g^{-1} h \in \pi_n(i) \pi_n(U_g, x_0)$. Since $\pi_n(i) \pi_n(U_g, x_0)$ is a subgroup of $\pi_n(X, x_0)$, $h^{-1} g \in \pi_n(i) \pi_n(U_g, x_0)$. Since $H$ is a subgroup of $\pi_n(X, x_0)$, $h^{-1} \in H$, and so $h^{-1} g \in Hg$. But we showed that $h^{-1} g$ is an element of $\pi_n(i) \pi_n(U_g, x_0) \cap Hg$ which is a contradiction to $\pi_n(i) \pi_n(U_g, x_0) \cap Hg = \emptyset$. Therefore, if $g \in \overline{H}$, then $g \not \in G -H$, that is $H$ is closed in $G$.

\item
($(ii) \Rightarrow (iii)$) Let $H$ be closed in $G$. Since $\pi_n^{wh}(X, x_0)$ is a left topological group, its subgroup $G$ is also a left topological group. Thus, by \cite[Theorem 1.5.1, p. 37]{ArhTka}, the coset space $\frac{G}{H}$ endowed with the quotient topology is a homogeneous $T_1$-space. Since each $T_1$-space is a $T_0$-space, the coset space $\frac{G}{H}$ is a $T_0$-space. By \cite[Theorem 3.4, p. 19]{BogSie}, part $((iii) \Rightarrow (i))$, the coset space $\frac{G}{H}$ is Hausdorff.

\item
$(iii) \Rightarrow (i))$ Let the coset space $\frac{G}{H}$ be Hausdorff. Then, for each $g \in G - H$, there exist open neighbourhoods $V$ and $W$ of $H$ and $Hg$, respectively, in $\frac{G}{H}$, such that $H \in V$  and $Hg \in W$, and $V \cap W = \emptyset$. Thus, $Hg \not \in V$, or equivalently, there is no $h \in H$ such that $hg \in q^{-1} (V)$, that is $Hg \cap q^{-1} (V) = \emptyset$, where $q :G \to \frac{G}{H}$ is the quotient map. Since $V$ is an open neighbourhood of $H$ in $\frac{G}{H}$, and $q$ is continuous, $q^{-1} (V)$ is an open neighbourhood of the identity in $G$. Hence, there exists an open neighbourhood $U$ of $x_0$ such that $\pi_n(i) \pi_n(U, x_0) \subseteq q^{-1} (V)$, where $i:U \to X$ is the inclusion map. Since $Hg \cap q^{-1} (V) = \emptyset$, and  $\pi_n(i) \pi_n(U, x_0) \subseteq q^{-1} (V)$, we can conclude that $Hg \cap \pi_n(i) \pi_n(U, x_0) = \emptyset$. Accordingly, for each $g \in G - H$, we can find an open neighbourhood $U$ of $x_0$ such that $\pi_n(i) \pi_n(U, x_0) \cap Hg = \emptyset$. Therefore, $X$ is $n$-homotopically Hausdorff relative to $(G, H)$ at $x_0$.
\end{enumerate}\end{proof}
Fisher et al. \cite[Lemma 2.10]{FisZas} proved that $X$ is homotopically Hausdorff at any point if and only if path space $\widetilde{X}$, containing $\pi_1^{wh}(X, x_0)$ as a subspace, is Hausdorff. Therefore, if $X$ is homotopically Hausdorff at any point, then $\pi_1^{wh}(X, x_0)$ is Hausdorff. The following corollary shows that the necessary and sufficient condition for $\pi_n^{wh}(X, x_0)$ to be Hausdorff is $n$-homotopically Hausdorffness of $X$ at $x_0$ for $n \ge 1$. Here, we give a special consequnce of Theorem \ref{th2.8}, when $G = \pi_n(X, x_0)$ and $H = \{e\}$.

\begin{corollary}\label{co2.9}
Let $(X, x_0)$ be a pointed space, and $n$ be a natural number. Then $X$ is $n$-homotopically Hausdorff at $x_0$ if and only if $\pi_n^{wh}(X, x_0)$ is Hausdorff.
\end{corollary}

The whisker topology on homotopy groups depends on the choice of the base point, and the structure of $\pi_n^{wh}(X, x_0)$ may differ even in a path component. Brodskiy et al. \cite[Corollary 4.9]{BroDyd} introduced some spaces, called small loop transfer spaces, on which the topological structure of $\pi_1^{wh}(X, x_0)$ homeomorphically transfers by all paths. Pashaei et al. \cite{PasMas} generalized \textit{small loop transfer} path (SLT path for abbreviation), which was introduced in \cite[Definition 4.7]{BroDyd}. In the following, we intend to extend this notion to higher dimensions. For this purpose, we need to recall the isomorphism $\Gamma_{\gamma}: \pi_n(X, x_0) \to \pi_n(X, x_1)$ induced by a path $\gamma$ from $x_0$ to $x_1$. See \cite[Page 381]{Spa}.

\begin{definition}\label{de3.5}
Let $\gamma$ be a path from $x_0$ to $x_1$ in $X$. Then  for any $n$-loop $\alpha$ at $x_0$, $\gamma_{\#} (\alpha)$  is defined to be an $n$-loop at $x_1$, where $\beta: (\mathbb{I}^n, \dot{\mathbb{I}^n}) \to (X, x_1)$   has the rule $\beta = \beta' \circ r$, in which $\beta' : (\mathbb{I}^n \times \{0\}) \cup ( \dot{\mathbb{I}^n} \times \mathbb{I}) \to X$ is defined by $\beta'(u, 0) =  \alpha(u)$ if $u \in \mathbb{I}^n$, and $\beta'( u, t) = \gamma (t)$ if $u \in \dot{\mathbb{I}^n}$ and $t \in \mathbb{I}$, and $r: \mathbb{I}^n \times \mathbb{I} \to (\mathbb{I}^n \times \{0\}) \cup ( \dot{\mathbb{I}^n} \times \mathbb{I})$ is the stereographic retraction.
\end{definition}

\begin{theorem}[\cite{Spa}]\label{th2.13}
Let $X$ be a space, and let $x_0, x_1 \in X$. For any path $\gamma$ from $x_0$ to $x_1$, there exists an isomorphism of groups
$\Gamma_{\gamma}: \pi_n(X, x_0) \to \pi_n(X, x_1)$ defined by $\Gamma_{\gamma} ([\alpha]) =[\gamma_{\#} (\alpha)]$.
\end{theorem}

The isomorphism $\Gamma_{\gamma}: \pi_n^{wh} (X, x_0) \to \pi_n^{wh} (X, x_1)$ is  not necessarily continuous, but for some paths called SLT paths, $\Gamma_{\gamma}$ is continuous (see \cite[Lemma 4.6]{BroDyd}). Now, we generalize SLT path to $n$-SLT path for $n \ge 1$, in order to make $\Gamma_{\gamma}$ continuous.

\begin{definition}
Let $X$ be a space, $x_0, x_1 \in X$, and $n \ge 1$. A path $\gamma$ from $x_0$ to $x_1$ is called small $ n $-loop transfer (abbreviated to $ n $-SLT),  if for every open neighbourhood $U$ of $ x_0 $, there exists an open neighbourhood $ V $ of $x_1$ such that for every $ n $-loop $\beta: (\mathbb{I}^n, \dot{\mathbb{I}^n}) \to (V, x_1) $, there is an $n$-loop $\alpha :(\mathbb{I}^n, \dot{\mathbb{I}^n}) \to (U, x_0) $ which is homotopic to $\gamma_{\#}^{-1}(\beta)$.
\end{definition}

Brodskiy et al. \cite[Lemma 4.6]{BroDyd} proved that $\gamma_{\#} :\pi_1^{wh} (X, x_0) \to \pi_1^{wh} (X, x_1)$ is continuous if and only if $\gamma^{-1}$ is a $1$-SLT path from $x_0$ to $x_1$ . The analogous assertion holds for $n \ge 2$ as follows.

\begin{proposition}\label{pr2.15}
Let $ \gamma $ be a path in $ X $ from $ x_0 $ to $x_1$. Then $\Gamma_{\gamma} : \pi^{wh}_n(X, x_0)  \to \pi^{wh}_n(X, x_1)$ is continuous  if and only if $ \gamma^{-1} $ is an $ n $-SLT path.
\end{proposition}

\begin{proof}
By Theorem \ref{th2.13} $\Gamma_{\gamma}:  \pi_n^{wh} (X, x_0) \to \pi_n^{wh}(X, x_1)$ is an isomorphism of groups. Since the whisker topology on the $n$th homotopy group makes it a left topological group, continuity of homomorphisms is equivalent to continuity at the identity \cite[Proposition 1.3.4, Page 19]{ArhTka}. Thus $\Gamma_{\gamma}$ is continuous if and only if it is continuous at the identity. By definition of the whisker topology, the set $\{\pi_n(i_2) \pi_n(V, x_1)|\ V$ is an open neighbourhood of $x_1\}$ is a local basis at the identity of $\pi_n^{wh} (X, x_1)$. Thus, $\Gamma$ is continuous at the identity if and only if for any open neighbourhood $V$ of $x_1$, $\Gamma_{\gamma}^{-1} \big( \pi_n(i_2) \pi_n(V, x_0)\big)$ is open in $\pi_n^{wh} (X, x_0)$. Again, since $\{\pi_n(i_1) \pi_n(U, x_0)|\ U$ is an open neighbourhood of $x_0\}$ is a local basis at the identity of $\pi_n^{wh}(X, x_0)$, $\Gamma_{\gamma}$ is continuous at the identity if and only if for any open neighbourhood $V$ of $x_1$, there is an open neighbourhood $U$ of $x_0$ such that $\Gamma_{\gamma} \big(\pi_n(i_1) \pi_n(U, x_0)\big) \subseteq \big( \pi_n(i_2) \pi_n(V, x_0)\big)$. That is for any $n$-loop $\alpha$ in $U$ at $x_0$, there is an $n$-loop $\beta$ in $V$ at $x_1$, such that $\Gamma_{\gamma} ([\alpha]) = [\beta]$. Since $\Gamma_{\gamma} ([\alpha]) = [\gamma_{\#} (\alpha)]$, $\beta$ is homotopic to $\gamma_{\#} (\alpha)$. Equivalently, since $\gamma_{\#} (\alpha) \simeq {\gamma^{-1}}_{\#}^{-1} (\alpha)$, $\beta$ is homotopic to ${\gamma^{-1}}_{\#}^{-1} ( \alpha)$. Therefore, $\Gamma_{\gamma}$ is continuous if and only if for any open neighbourhood $V$ of $x_1$, there is an open neighbourhood $U$ of $x_0$, such that for every $n$-loop $\alpha$ in $U$ at $x_0$, there is an $n$-loop $\beta$ in $V$ at $x_1$ homotopic to ${\gamma^{-1}}_{\#}^{-1} (\alpha)$, or equivalently, $\gamma^{-1}$ is an $n$-SLT path from $x_1$ to $x_0$.
\end{proof}

Proposition \ref{pr2.15} implies the following corollary.

\begin{corollary}\label{co2.7}
Let $X$ be a space, $x_0, x_1 \in X$, and $n \ge 2$. If there is a path $ \gamma $ from $ x_0 $ to $ x_1 $ such that $ \gamma $ and $ \gamma^{-1} $ are $ n $-SLT paths, then $ \pi^{wh}_n(X,x_0) $ and $ \pi^{wh}_n(X,x_1) $ are isomorphic as topological groups.
\end{corollary}

\section{Relationship Between $L_n(X, x_0)$ and $\pi_n^{wh} (X,x_0)$}\label{sec4}

Let $\varphi: \mathcal{H}_n(X, x_0) \to \prod_{\aleph_0} \pi_n(X, x_0)$ be the homomorphism (I). In this section, we study the homomorphic image of the Hawaiian group, by the homomorphism $\varphi$, and its relation to the whisker topology on homotopy groups.

 It is shown that $L_n(X,x_0)$, introduced in \cite[Definition 2.6]{1}, is equal to $Im (\varphi)$ and hence it is a subgroup of $\prod_{\aleph_0}\pi_n(X,x_0)$, for each pointed space $(X,x_0)$. Note that the structure of $L_n(X,x_0)$ does not depend only on $\pi_n(X,x_0)$. In the following example, for every $n \ge 1$ we present two  pointed spaces $(X, x_0)$ and $(Y, y_0)$  with $\pi_n (X,x_0) \cong \pi_n(Y, y_0)$, but $L_n(X,x_0) \not \cong L_n(Y,y_0)$.

\begin{example}\label{ex0.5}
Let $n\ge 1$. Put $Y$ the Eilenberg-MacLane space with $\pi_n(Y,y_0) \cong \prod_{\aleph_0} \mathbb{Z}$ and $X= \prod_{\aleph_0} \mathbb{S}^n$. If $x_0 \in X$, then
\[
\pi_n(X,x_0) \cong \prod_{\aleph_0} \pi_n(\mathbb{S}^n, 1) \cong \prod_{\aleph_0} \mathbb{Z} \cong \pi_n(Y,y_0).
\]
Since $Y$ is locally $n$-simply connected at $y_0$, $\mathcal{H}_n(Y,y_0) \cong L_n(Y,y_0) \cong \prod_{\aleph_0}^W \pi_n(Y,y_0)$ (see \cite[Theorem 1]{KarRep}), and therefore, $L_n(Y,y_0) \cong \prod_{\aleph_0}^W \prod_{\aleph_0} \mathbb{Z}$.

By a straightforward argument, one can prove that $L_n$ preserves the products, for all $n \ge 1$. Thus,  $L_n(X,x_0) \cong \prod_{\aleph_0} L_n(\mathbb{S}^n, 1)$. Since $\mathbb{S}^n$ is locally $n$-simply connected at $1$, $\mathcal{H}_n(\mathbb{S}^n, 1) \cong L_n(\mathbb{S}^n, 1) \cong \prod_{\aleph_0}^W \pi_n(\mathbb{S}^n , 1)$ (see \cite[Theorem 1]{KarRep}). Therefore, $L_n(X,x_0) \cong \prod_{\aleph_0} \prod_{\aleph_0}^W \mathbb{Z}$.

Note that $\prod_{\aleph_0} \prod_{\aleph_0}^W \mathbb{Z} \not \cong \prod_{\aleph_0}^W \prod_{\aleph_0} \mathbb{Z}$ (see \cite{Zim}), and hence, $L_n(X,x_0) \not \cong L_n(Y, y_0)$.
\end{example}

Example \ref{ex0.5} shows that the algebraic structure of $\pi_n(X, x_0)$ does not determine the structure of $L_n(X,x_0)$. But in Theorem \ref{th2.3n}, we will see that the whisker topology on $\pi_n(X, x_0)$ can exactly characterize $L_n(X,x_0)$. The following theorem manifests the relation between $L_n(X, x_0)$ and $\pi_n^{wh}(X, x_0)$.

\begin{theorem}\label{th2.3}
Let $(X, x_0)$ be a pointed space and $n \ge 1$.
Then $L_n(X, x_0)$ is equal to the set of all sequences converging to the identity in $\pi_n^{wh} (X,x_0)$.
\end{theorem}

\begin{proof}
A sequence $\{[\alpha_k]\}_{\aleph_0}$ belongs to $L_n(X,x_0)$ if and only if there exists null-convergent sequence $\{\beta_k\}_{\aleph_0}$ with $\alpha_k \simeq \beta_k$ for every $k \in \mathbb{N}$. A sequence $\{\beta_k\}_{\aleph_0}$ is null-convergent if and only if for each open set $U$ of $x_0$ there exists $K\in \mathbb{N}$ such that if $k \ge K$, then $im (\beta_k) \subseteq U$. Recall that $im (\beta_k) \subseteq U$ if and only if there exists $\gamma:(\mathbb{S}^n,1) \to (U, x_0)$ such that $\beta_k \simeq i \circ \gamma$, where $i: U \to X$ is the inclusion map.  Hence, $\{\beta_k\}_{\aleph_0}$ is null-convergent if and only if there exists $K \in \mathbb{N}$ such that if $k \ge K$, then $[\beta_k] \in \{[i \circ \gamma]| \gamma$ is an $n$-loop at $x_0$ in $U\} = \pi_n(i) \pi_n(U,x_0)$, or equivalently $[\alpha_k] \in \pi_n (i) \pi_n(U, x_0)$.

Therefore, $\{[\alpha_k]\}_{\aleph_0} \in L_n(X,x_0)$ if and only if for each open set $U$ of $x_0$, there exists $K\in \mathbb{N}$ such that if $k \ge K$, then $[\alpha_k] \in \pi_n(i) \pi_n(U,x_0)$. Since the set $\{\pi_n(i) \pi_n(U,x_0)|\ U$ is an open subset of  $x_0 \}$ forms a local basis for the whisker topology on $\pi_n (X,x_0)$ at the identity, $\{[\alpha_k]\}_{\aleph_0} \in L_n(X,x_0)$ if and only if $\{[\alpha_k]\}_{\aleph_0}$ converges to the identity in $\pi_n^{wh}(X, x_0)$.
\end{proof}

Recall that by the definition of whisker topology on the $n$th homotopy group of pointed space $(X, x_0)$, $\pi_n^{wh}(X, x_0)$ is indiscrete if and only if all $n$-loops in $X$ at $x_0$ are small. Also by Proposition 3.1 $\pi_n^{wh}(X, x_0)$ is discrete if and only if $X$ is semi-locally $n$-simply connected at $x_0$.

\begin{corollary}\label{co3.3n}
Let $X$ be a space having a countable local basis at $x_0$.
\begin{enumerate}
\item
$X$ is semi-locally $n$-simply connected at $x_0$ if and only if
$L_n(X, x_0) = \prod^W_{\aleph_0} \pi_n(X, x_0)$.
\item
All $n$-loops at $x_0$ are small if and only if $L_n(X, x_0) = \prod_{\aleph_0} \pi_n(X, x_0)$.
\end{enumerate}
\end{corollary}

\begin{proof}
Since $X$ has a countable local basis at $x_0$, $\pi_n^{wh}(X, x_0)$ is first countable, by Lemma \ref{le2.3}. %Therefore
\begin{enumerate}
\item
 $\pi_n^{wh} (X, x_0)$ is discrete if and only if every convergent sequence is eventually constant. Since $\pi_n^{wh}(X, x_0)$ is a left topological group, every convergent sequence is obtained by some left translation from a sequence converging to the identity. Hence by Theorem \ref{th2.3} the result holds.

\item
If $\pi_n^{wh} (X, x_0)$ is indiscrete, then all sequences are convergent. Hence, all sequences in $\pi_n^{wh}(X, x_0)$ converge to the identity.  By Theorem \ref{th2.3}, $L_n(X, x_0)$ equals the set of convergent sequences to the identity of $\pi_n^{wh}(X, x_0)$, and then $L_n(X, x_0) = \prod_{\aleph_0} \pi_n(X, x_0)$.

Conversely, if $L_n(X, x_0) = \prod_{\aleph_0} \pi_n(X, x_0)$, then all sequences converge to the identity in $\pi_n^{wh}(X, x_0)$. It is equivalent to $\pi_n^{wh}(X, x_0)$ be indiscrete at the identity. Since $\pi_n^{wh}(X, x_0)$ is a left topological group, it is indiscrete at every point.
\end{enumerate}
\end{proof}
Note that the $n$-Hawaiian earring space, $\mathbb{HE}^n$, does not belong to the two classes of Corollary \ref{co3.3n}, and hence $\pi_n^{wh} (\mathbb{HE}^n,  \theta_0)$ is not discrete nor indiscrete. The $n$-Hawaiian earring space was generalized to $n$-Hawaiian like spaces by Ghane et al. \cite{Ghane} as a specified topology on disjoint union of CW spaces with a common point as follows.
\begin{definition}[\cite{Ghane}]
Let $\{X_i\}_{i \in \mathbb{N}}$ be a family of topological spaces. Suppose that the underlying set of $\widetilde{\bigvee}_{i \in \mathbb{N}} X_i$ is the disjoint union of $X_i$'s with exactly one point $x_*$ in common, equipped with a topology generated by the neighbourhood bases as follows.
\begin{enumerate}
\item
If $x \in X_i \setminus \{x_*\}$, then the neighbourhood basis of $\widetilde{\bigvee}_{i \in \mathbb{N}} X_i$ at $x$ is the one of $X_i$, $i \in \mathbb{N}$.
\item
At point $x_*$, the neighbourhood basis consists of sets of the form $\bigcup_{i \in \mathbb{N} \setminus F} X_i \cup \bigcup_{i \in F} U_i$, where $F$ is a finite set of natural numbers and $U_i$ is an open neighbourhood of $x_*$ in $X_i$.
\end{enumerate}
The space $\widetilde{\bigvee}_{i \in \mathbb{N}} X_i$ is called an $n$-Hawaiian like space, when $X_i$'s are all $(n-1)$-connected compact CW spaces.
\end{definition}
Let $\pi_n^{qtop}(X, x_*)$ denote the quasi-topological $n$th homotopy group induced by the compact-open topology on the $n$-loop space $\Omega_n(X, x_*)$ (see \cite{GhaMas}). If $X= \widetilde{\bigvee}_{i \in \mathbb{N}} X_i$ is an $n$-Hawaiian like space, then for $n \ge 2$, it was shown in \cite[Theorem 1.1]{Ghane} that $\pi_n(X, x_*) \cong \prod_{i \in \mathbb{N}} \pi_n(X_i, x_*)$ and that $\pi_n^{qtop} (X, x_*)$ is isomorphic to the prodiscrete topological group $\prod_{i \in \mathbb{N}} \pi_n(X_i, x_*)$. In \cite[Theorem 3.3]{Ghane}, Ghane et al. proved that the topologies of $\pi_n^{lim}(X, x_*)$ and $\pi_n^{qtop}(X, x_*)$ coincide if $X$ is an $n$-Hawaiian like space.

Let $\{X_i\}_{i \in \mathbb{N}}$ be a family of spaces each of which is Tychonoff, $(n-1)$-connected, locally strongly contractible and first countable at $x_i$.
We call $X = \widetilde{\bigvee}_{i \in \mathbb{N}} X_i$, the compact union of the above family, the generalized $n$-Hawaiian like space. In \cite[Theorem 1.1]{EdaKaw}, it was proved that for $n \ge 2$, $\pi_n(X, x_*) \cong \prod_{\mathbb{N}} \pi_n(X_i, x_*)$.
In  the following proposition, we show that for generalized $n$-Hawaiian like spaces, the topology of $\pi_n^{wh} (X, x_*)$ is prodiscrete.

\begin{proposition}\label{pr4.5}
If $X = \bigvee_{i \in \mathbb{N}} X_i$ is a generalized $n$-Hawaiian like space and $x_*$ is the common point, $n \ge 2$, then $\pi_n^{wh}(X, x_*)$ is isomorphic to the prodiscrete topological group $\prod_{i \in \mathbb{N}} \pi_n(X_i, x_*)$.
\end{proposition}

\begin{proof}
Using the isomorphism $\pi_n(X, x_*) \cong \prod_{\mathbb{N}} \pi_n(X_i, x_*)$, we can consider elements of $\pi_n(X, x_*)$ by the corresponding ones of $\prod_{i \in \mathbb{N}} \pi_n(X_i, x_*)$. That is $[f] \in \pi_n (X, x_*)$ can be considered as $([f^1], [f^2], \ldots) \in \prod_{i \in \mathbb{N}} \pi_n(X_i, x_*)$, where $f^i = r_i \circ f$ and $r_i :X \to X_i$ is the natural retraction.
Since $X$ is first countable and $n$-homotopically Hausdorff at $x_*$, then $\pi_n^{wh} (X, x_*)$ is a metric topological group by Corollary \ref{co2.11}. Thus, the topology of $\pi_n^{wh}(X, x_*)$ is identified by convergent sequences. By Theorem \ref{th2.3}, the set of convergent sequences to the identity of $\pi_n^{wh}(X, x_*)$ is equal to $L_n(X, x_*)$.
It suffices to verify that $\{([f_k^1], [f_k^2], \ldots)\}_{k \in \mathbb{N}} \in L_n(X, x_*)$ if and only if it converges to the identity in prodiscrete topological group $\prod_{i \in \mathbb{N}} \pi_n(X_i, x_*)$. Let $\{([f_k^1], [f_k^2], \ldots)\}_{\mathbb{N}} \in L_n(X, x_*)$. We must show that for any open set $U$ of the identity in $\prod_{i \in \mathbb{N}} \pi_n(X_i, x_*)$, $([f_k^1], [f_k^2], \ldots) \in U$ for all $k \in \mathbb{N}$ except a finite number.
%Since $\pi_n^{qtop} (X, x_*)$ is isomorphic to prodiscrete topological group $\prod_{i \in \mathbb{N}} \pi_n(X_i, x_*)$,
The elements of the local basis at the identity of $\prod_{i \in \mathbb{N}} \pi_n(X_i, x_*)$  are of the
form $U_i=\{e_1\} \times \{e_2\} \times \cdots \times \{ e_{i-1}\} \times \pi_n(X_i, x_*) \times \pi_n( X_{i+1}, x_*) \times \cdots$, for some $i \in \mathbb{N}$, where $e_1, e_2, \ldots$ are the identity elements of $\pi_n(X_1, x_*), \pi_n(X_2, x_*), \ldots$, respectively.
By \cite[Proof of Theorem 2.10]{1}, $\{([f_k^1], [f_k^2], \ldots)\}_{k \in \mathbb{N}} \in L_n( X, x_*)$ if and only if $[f_k^j]$ is the identity element for all $j \in \mathbb{N}$ except a finite number. Thus, for every $j \in \mathbb{N}$, there exists $K_j \in \mathbb{N}$ such that if $k \ge K_j$, then $[f_k^j] = e_j$. Put $K = \max \{K_1, K_2, \ldots, K_{i-1}\}$. If $k \ge K$, then $[f_k^j] = e_j$, for $j <i$. Therefore, $([f_k^1], [f_k^2], \ldots) \in U_i$ if $k \ge K$. That is $\{([f_k^1], [f_k^2], \ldots)\}_{k \in \mathbb{N}}$ converges to the identity in $\prod_{i \in \mathbb{N}} \pi_n(X_i, x_*)$.

Conversely, let $\{([f_k^1], [f_k^2], \ldots)\}_{k \in \mathbb{N}}$ converges to the identity in $\prod_{i \in \mathbb{N}} \pi_n(X_i, x_*)$. By the form of the local basis at the identity in $\prod_{i \in \mathbb{N}} \pi_n(X_i, x_*)$, $U_i$'s, there exists $K_{i+1} \in \mathbb{N}$ such that if $k \ge K_{i+1}$, then $[f_k^j] = e_j$ for $j  \le i$. Equivalently, $[f_k^i] $'s are identity element except possibly finite numbers $k < K_{i+1}$. Again, by \cite[Proof of Theorem 2.10]{1}, the sequence $\{([f_k^1], [f_k^2], \ldots)\}_{k \in \mathbb{N}}$ belongs to $L_n( X, x_*)$.
\end{proof}

\begin{example}\label{ex2.4}
For the $n$-dimensional Hawaiian earring, $\mathbb{HE}^n$, it is proved that $\pi_n( \mathbb{HE}^n) \cong \prod_{\mathbb{N}} \mathbb{Z}$ \cite[Corollary 1.2]{EdaKaw}.  Then $\pi_n^{wh}(\mathbb{HE}^n, \theta)$ is isomorphic to the prodiscrete topological group of  $\prod_{\mathbb{N}} \mathbb{Z}$.
\end{example}

Theorem \ref{th2.3} shows that there exists a close relation between $L_n(X,x_0)$ and $\pi_n^{wh} (X,x_0)$. In the following theorem, we prove that the structure of $\pi_n^{wh}(X,x_0)$ fixes the structure of $L_n(X,x_0)$. Note that an isomorphism of left topological groups is an isomorphism of groups which is also a homeomorphism on the underlying topological space.
\begin{theorem}\label{th2.3n}
Let $(X, x_0)$ and $(Y, y_0)$ be two pointed spaces and let $n \ge 1$. If $\pi_n^{wh} (X, x_0) \cong \pi_n^{wh} (Y, y_0)$ as left topological groups, then $L_n(X, x_0) \cong L_n(Y, y_0)$.
Moreover, if $X$ and $Y$ have countable local bases at $x_0$ and $y_0$, respectively, and if the isomorphism $L_n(X,x_0) \cong L_n(Y,y_0)$ is induced by some isomorphism $g: \pi_n(X,x_0) \to \pi_n(Y,y_0)$, then $g$ is a homeomorphism.
\end{theorem}

\begin{proof}
Let $g: \pi_n^{wh} (X,x_0) \to \pi_n^{wh} (Y,y_0)$ be an isomorphism of left topological groups. Since $g$ is an isomorphism from $\pi_n(X, x_0)$ onto $\pi_n(Y, y_0)$, it induces monomorphisms $\widetilde{g}:L_n(X,x_0) \to \prod_{\aleph_0} \pi_n(Y,y_0)$ and $\widetilde{g^{-1}}: L_n(Y, y_0) \to \prod_{\aleph_0} \pi_n(X,x_0)$ by the rule $\widetilde{g} (\{[\alpha_k]\}_{\aleph_0}) = \{g([\alpha_k])\}_{\aleph_0}$ and $\widetilde{g^{-1}} (\{[\beta_k]\}_{\aleph_0}) = \{g^{-1}([\beta_k])\}_{\aleph_0}$, respectively. We show that $\widetilde{g} (L_n(X,x_0)) \subseteq L_n(Y,y_0)$ and $\widetilde{g^{-1}} (L_n(Y, y_0)) \subseteq L_n(X, x_0)$.
Let $\{[\alpha_k]\}_{\aleph_0} \in L_n(X, x_0)$, then by Theorem \ref{th2.3}, $\{[\alpha_k]\}_{\aleph_0}$ converges to the identity in $\pi_n^{wh} (X, x_0)$. Since $g$ is a continuous homomorphism, $\{g([\alpha_k])\}_{\aleph_0}$ converges to the identity in $\pi_n^{wh}(Y, y_0)$. By Theorem \ref{th2.3}, $\widetilde{g} (\{[\alpha_k]\}_{\aleph_0}) =\{g([\alpha_k])\}_{\aleph_0} \in L_n(Y, y_0)$. Since $\{[\alpha_k]\}_{\aleph_0}$ is an arbitrary element of $L_n(X, x_0)$, it implies that $\widetilde{g} (L_n(X,x_0)) \subseteq L_n(Y,y_0)$. A similar argument can be applied to show that $\widetilde{g^{-1}} (L_n(Y, y_0)) \subseteq L_n(X, x_0)$. Moreover,
\[
\widetilde{g} \circ \widetilde{g^{-1}} \big(\{[\beta_k]\}_{\aleph_0}\big) = \widetilde{g} \big(\{ g^{-1} [\beta_k]\}_{\aleph_0}\big) = \{g \circ g^{-1} [\beta_k]\}_{\aleph_0} = \{[\beta_k]\}_{\aleph_0}.
\]
Hence $\widetilde{g} \circ \widetilde{g^{-1}} = id_{L_n(Y, y_0)}$. Similarly $\widetilde{g^{-1}} \circ \widetilde{g} = id_{L_n(X, x_0)}$. Therefore $\widetilde{g}: L_n(X, x_0) \cong L_n(Y,y_0)$.

Conversely, let $g: \pi_n(X, x_0) \to \pi_n(Y,y_0)$ be the isomorphism inducing $h:L_n(X,x_0) \cong L_n(Y, y_0)$ by the rule $h (\{[\alpha_k]\}_{\aleph_0}) = \{g([\alpha_k])\}_{\aleph_0}$. We must show that $g$ and $g^{-1}$ are continuous. Since $g$ and $g^{-1}$ are homomorphisms and also $\pi_n^{wh}(X, x_0)$ and $\pi_n^{wh}(Y, y_0)$ are left topological groups, $g$ and $g^{-1}$ are continuous if they are continuous at the identities by \cite[Proposition 1.3.4]{ArhTka}. Moreover, since $X$ and $Y$ are first countable at $x_0$ and $y_0$, respectively, $\pi_n^{wh} (X,x_0)$ and $\pi_n^{wh}(Y,y_0)$ are first countable spaces. Thus, to prove continuity of $g$ and $g^{-1}$, it suffices to check sequential continuity at the identities. Let $\{[\alpha_k]\}_{\aleph_0}$ be a sequence converges to the identity in $\pi_n^{wh}(X, x_0)$. By Theorem \ref{th2.3}, $\{[\alpha_k]\}_{\aleph_0} \in L_n(X, x_0)$. Since $h(L_n(X, x_0)) \subseteq L_n(Y, y_0)$, $h$ maps $\{[\alpha_k]\}_{\aleph_0}$ into $L_n(Y, y_0)$. Again by Theorem \ref{th2.3}, $h(\{[\alpha_k]\}_{\aleph_0})$ converges to the identity in $\pi_n^{wh}(Y, y_0)$. Moreover, $h(\{[\alpha_k]\}_{\aleph_0}) = \{g([\alpha_k])\}_{\aleph_0}$. Therefore, $\{g([\alpha_k])\}_{\aleph_0}$ converges to the identity. Since $g$ is a homomorphism, $g$ maps the identity of $\pi_n^{wh}(X, x_0)$ to the identity element of $\pi_n^{wh}(Y, y_0)$. Thus, $g$ is sequentially continuous at the identity as required.
Similarly, by using the inclusion $h^{-1} \big(L_n(Y,y_0)\big) \subseteq L_n(X, x_0)$, one can show that $g^{-1}$ is continuous.
Thus, $g$ and $g^{-1}$ are continuous maps, and hence $g$ is a homeomorphism.
\end{proof}

Let $x_0, x_1 \in X$. If there exists a path $\gamma$ from $x_0$ to $x_1$, then $\gamma_\#$ in Definition \ref{de3.5} induces an isomorphism from $\pi_n(X, x_0)$ onto $\pi_n(X,x_1)$.
But there exist path connected spaces, namely $\mathbb{HE}^n$, $n \ge 2$, such that $L_n (\mathbb{HE}^n, \theta) \not \cong L_n( \mathbb{HE}^n, a)$, where $a \neq \theta$ (see \cite[Corollary 2.11]{1}). By Theorem \ref{th2.3n} and Corollary \ref{co2.7}, $\gamma_\#$ can analogously transfer $L_n(X, x_0)$ isomorphically onto $L_n(X, x_1)$, if $\gamma$ and $\gamma^{-1}$ are $n$-SLT paths.

\begin{corollary}\label{co3.5}
Let $X$ have countable local bases at two points $x_0$ and $x_1$, and $n\ge 1$. If there exists  a path $\gamma$  from $x_0$ to $x_1$, such that $\gamma$ and $\gamma^{-1}$ are $n$-SLT paths, then $\{\Gamma_{\gamma}\}_{\aleph_0}: L_n(X, x_0) \to L_n(X,x_1)$ is an isomorphism.
\end{corollary}

By Corollary \ref{co3.5}, if $\varphi: \mathcal{H}_n(X, x_0)\to L_n(X, x_0)$ is injective, and $\gamma$ and $\gamma^{-1}$ are $n$-SLT paths, then $\{\Gamma_{\gamma}\}_{\aleph_0}$ induces an isomorphism from $\mathcal{H}_n(X, x_0)$ onto $\mathcal{H}_n(X, x_1)$. For instance, on semilocally $n$-simply connected spaces, we have such an isomorphism.

The harmonic archipelago, $\mathbb{HA}$, is a non-simply connected space with small loops. The fundamental group and homology groups of the harmonic archipelago were studied in \cite{ConHoj} and \cite{KarRep3}, respectively. Here, we recall some of their results to use in Example \ref{ex4.12n}.

\begin{figure}[!ht]
\centering
\includegraphics[scale=.4]{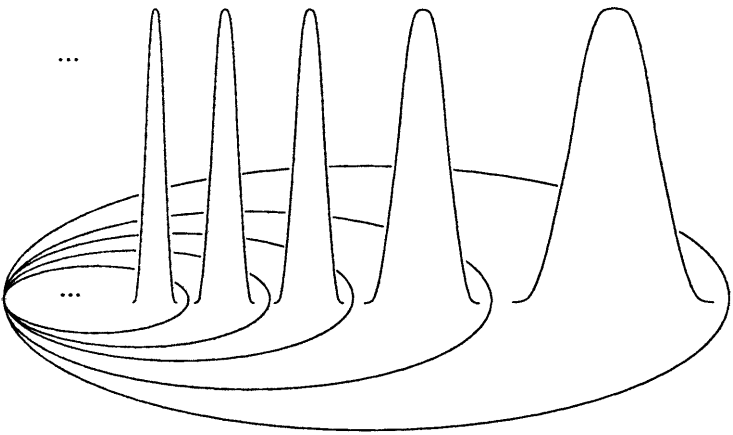}
\end{figure}

\begin{theorem}[\cite{ConHoj, KarRep3}]\label{th4.12n}
Let $\times^{\sigma}$ denote the free $\sigma$-product of a family of groups, and $\overline{H}^N$ denote the normal closure of the subgroup $H$ in a given group. Then\\
%\begin{enumerate}
%\item
$(i)$ \cite[Theorem 5]{ConHoj}
\[\pi_1 (\mathbb{HA}) \cong \frac{\times^{\sigma}_{\aleph_0} \mathbb{Z}}{\overline{*_{\aleph_0} \mathbb{Z}}^N}.
\]
%\item
$(ii)$ \cite[Theorem 1.2 and Proposition 2.4]{KarRep3}. Let $P$ be the set of all prime numbers. Then
\[
H_1 (\mathbb{HA}) \cong \frac{\prod_{\aleph_0} \mathbb{Z}}{\sum_{\aleph_0} \mathbb{Z}} \cong \left( \prod_{p\in P} A_p \right) \oplus \left( \sum_{c} \mathbb{Q} \right),
\]
where $A_p$ is the $p$-adic completion of the direct sum of $p$-adic integers $\sum_{c} \mathbb{J}_p$, and $c$ denotes the continuum cardinal.
%\end{enumerate}
\end{theorem}

Example \ref{ex4.12n} illustrates that Corollary \ref{co3.5} does not hold if there is no such path between the points.

\begin{example}\label{ex4.12n}
Let $\mathbb{HA}$ be the harmonic archipelago space, and $\theta$ be the origin.

Let $a \in \mathbb{HA}$ and $a \neq \theta$. Then by Corollary \ref{co3.3n}, $L_1(\mathbb{HA}, a) = \prod_{\aleph_0}^W \pi_1(\mathbb{HA}, a)$ and $L_1(\mathbb{HA}, \theta) = \prod_{\aleph_0} \pi_1(\mathbb{HA}, \theta)$. By Theorem \ref{th4.12n} (i) since $\pi_1(\mathbb{HA}) \cong \frac{\times^{\sigma}_{\aleph_0} \mathbb{Z}}{\overline{*_{\aleph_0} \mathbb{Z}}^N}$, we have
\[
L_1(\mathbb{HA}, a) \cong {\prod_{\aleph_0}}^W \frac{\times_{\aleph_0}^{\sigma} \mathbb{Z}}{\overline{*_{\aleph_0} \mathbb{Z}}^N}, \quad
L_1(\mathbb{HA}, \theta) \cong \prod_{\aleph_0} \frac{\times_{\aleph_0}^{\sigma} \mathbb{Z}}{\overline{*_{\aleph_0} \mathbb{Z}}^N}.
\]
We prove that $L_1(\mathbb{HA}, a) \not \cong L_1(\mathbb{HA}, \theta)$.
By contrary, assume that $L_1(\mathbb{HA}, a) \cong L_1(\mathbb{HA}, \theta)$.  Thus, their  abelianizations must be isomorphic.
That is $Ab(\prod_{\aleph_0}^W \frac{\times_{\aleph_0}^{\sigma} \mathbb{Z}}{\overline{*_{\aleph_0} \mathbb{Z}}^N}) \cong \sum_{\aleph_0} Ab(\frac{\times_{\aleph_0}^{\sigma} \mathbb{Z}}{\overline{*_{\aleph_0} \mathbb{Z}}^N}) \cong Ab (\prod_{\aleph_0} \frac{\times_{\aleph_0}^{\sigma} \mathbb{Z}}{\overline{*_{\aleph_0} \mathbb{Z}}^N})$.
Let $G = \prod_{\aleph_0} \frac{\times_{\aleph_0}^{\sigma} \mathbb{Z}}{\overline{*_{\aleph_0} \mathbb{Z}}^N}$. Then $G$ is  the fundamental group of the countably infinite product of copies of the
harmonic archipelago. Thus $G$ is the fundamental group of a space $X$ in which each
based loop has arbitrarily small representatives. Then by \cite[Theorem 4]{HerHoj}, we know $G$ satisfies the property of being Higman-complete. Moreover, the first singular
homology $H_1(X)$ is isomorphic to the abelianization of $G$.
Since we are assuming that $G$ is isomorphic to the weak direct
product $\prod_{\aleph_0}^W \frac{\times_{\aleph_0}^{\sigma} \mathbb{Z}}{\overline{*_{\aleph_0} \mathbb{Z}}^N}$ and the abelianization of a weak direct product can be computed
coordinatewise, we get that the abelianization of $G$ is isomorphic to $\sum_{\aleph_0} \bigg(\big( \prod_{p\in P} A_p \big) \oplus \big( \sum_{c} \mathbb{Q} \big)\bigg)$. In particular, the abelianization of $G$ is torsion-free. Then by \cite[Corollary 5]{HerHoj},
since $Ab(G) \cong H_1(X)$ is torsion-free it must be
algebraically compact. Now $\sum_{\aleph_0} \big( \prod_{p\in P} A_p \big)$ is algebraically compact as a direct summand
of the algebraically compact abelian group $Ab(G)$.
Moreover,
the group $A_p$ is the $p$-adic completion of $\sum_{c} \mathbb{J}_p$, and thus it is complete in $p$-adic topology. By \cite[p. 163, Remark]{Fuc}, since $p$-adic topology is coarser than $\mathbb{Z}$-adic topology, $A_p$ is reduced algebraically compact. By \cite[p. 101, Exercise 5]{Fuc}, a direct sum or a direct product of groups is reduced if and only if every component is reduced. Therefore, $\sum_{\mathbb{N}} \prod_{ p \in P} A_p$ is reduced algebraically compact. By
\cite[p. 163, Theorem 19.1]{Fuc}, a group is complete in the $\mathbb{Z}$-adic topology if and only if it is reduced algebraically compact.
Thus, $\sum_{\mathbb{N}} \prod_{ p \in P} A_p$ is complete in $\mathbb{Z}$-adic topology. Also, by \cite[p. 166, Corollary 39.10]{Fuc} if $A = \sum_{i \in I} C_i$ is a direct decomposition of a complete group $A$, then all the $C_i$ are complete groups, and there is an integer $n >0$ such that $n C_i =0$ for almost all $i \in I$.
Hence, there is an integer $n >0$ such that $n \prod_{ p \in P} A_p = 0$. It is equivalent to $\prod_{ p \in P} A_p$ being torsion group, which is a contradiction. Therefore, there is no isomorphism from $L_1( \mathbb{HA}, a)$ onto $L_1( \mathbb{HA}, \theta)$.

Note that $\pi_1^{wh}(\mathbb{HA}, a)$ is isomorphic to the discrete topological group $\frac{\times^{\sigma}_{\aleph_0} \mathbb{Z}}{\overline{*_{\aleph_0} \mathbb{Z}}^N}$, and $\pi_1^{wh}(\mathbb{HA}, \theta)$ is isomorphic to indiscrete one. Hence,
 there is no isomorphism of left topological groups from $\pi_1^{wh}(\mathbb{HA}, a)$ onto $\pi_1^{wh}(\mathbb{HA}, \theta)$, but one can not deduce that $L_1(\mathbb{HA}, a) \not \cong L_1(\mathbb{HA}, \theta)$.
\end{example}

\acknowledgment{The authors would like to thank the referees for their careful reading and useful comments and suggestions that helped to improve the article.} %If necessary


\begin{thebibliography}{99}

\bibitem{paper4}
M. Abdullahi Rashid, N. Jamali, B. Mashayekhy, S.Z. Pashaei and H. Torabi, {\it On subgroup topologies on fundamental groups}, to appear in Hacettepe J. Math. Stat.

\bibitem{ArhTka}
A. Arhangegel'skii and M. Tkachenko, {\em Topological Groups and Related Structures}, Atlantis Press, Amsterdam, 2008.

\bibitem{1} A. Babaee, B. Mashayekhy and H. Mirebrahimi, {\it On Hawaiian groups of some topological spaces}, Topology Appl.  \textbf{159} (8), 2043-2051, 2012.

\bibitem{BogSie}
W.A. Bogley and A.J. Sieradski, {\it Universal path spaces}, http: // people. oregonstate. edu/ ~bogleyw/ research/ ups.pdf.

\bibitem{BroDyd2}
N. Brodskiy, J. Dydak, B. Labuz and A. Mitra, {\it Covering maps for locally path connected spaces}, Fund. Math.  \textbf{218}, 13-46, 2012.

\bibitem{BroDyd}
N. Brodskiy, J. Dydak, B. Labuz and A. Mitra, {\it Topological and uniform structures on universal covering spaces}, arXiv:1206.0071, 2012.

\bibitem{ConHoj}
G.R. Conner, W. Hojka and M. Meilstrup, {\it Archipelago groups}, Proc. Amer. Math. Soc. \textbf{143}, 4973-4988, 2015.

\bibitem{ConLam}
G.R. Conner and J. Lamoreaux, {\it On the existence of universal covering spaces for metric
spaces and subsets of the Euclidean plane}, Fund. Math. \textbf{187}, 95-110, 2005.

\bibitem{EdaKaw}
K. Eda and K. Kawamura, {\it Homotopy and homology groups of the $n$-dimensional Hawaiian Earring},
Fund. Math.  \textbf{165} (1), 17-28, 2000.

\bibitem{Fab}
P. Fabel, {\it Multiplication is discontinuous in the Hawaiian Earring droup (with the Quotient
Topology)}, Bull. Pol. Acad. Sci. Math. \textbf{59} (1), 77-83, 2011.

\bibitem{FisZas}
H. Fischer and A. Zastrow, {\it Generalized universal coverings and the shape group}, Fund. Math. \textbf{197}, 167-196, 2007.

\bibitem{Fuc}
L. Fuchs, {\em Infinite Abelian Groups I}, Academic Press, New York, 1970.

\bibitem{Ghane}
 F.H. Ghane, Z. Hamed, B. Mashayekhy and H. Mirebrahimi, {\it On topological homotopy groups of $n$-Hawaiian like spaces}, Topology Proceedings  \textbf{36}, 255--266, 2010.

\bibitem{GhaMas}
H. Ghane, Z. Hamed, B. Mashayekhy, and H. Mirebrahimi, {\it Topological
homotopy groups}, Bull. Belg. Math. Soc. Simon Stevin \textbf{15} (3), 455-464, 2008.

\bibitem{HerHoj}
Herfort and Hojka, {\it Cotorsion and wild homology}, Israel J. Math. \textbf{221},  275-290, 2017.

\bibitem{paper3}
N. Jamali, B. Mashayekhy, H. Torabi, S.Z. Pashaei and M. Abdullahi Rashid,
{\it On topologized fundamental groups with small loop transfer viewpoints}, Acta Math. Vietnamica \textbf{43}, 1-27, 2018.

\bibitem{KarRep}
U.H. Karimov and  D. Repov\v{s}, {\it Hawaiian groups of topological spaces (Russian)},  Uspekhi. Mat.
Nauk.  \textbf{61} (5), 185--186, 2006; transl. in Russian Math. Surv. \textbf{61} (5), 987--989, 2006.

\bibitem{KarRep3}
U.H. Karimov and  D. Repov\v{s}, {\it On the homology of the Harmonic archipelago}, Central European J. Math. \textbf{10}, 863--872, 2012.

\bibitem{PasMas}
S.Z. Pashaei, B. Mashayekhy, H. Torabi and M. Abdullahi Rashid,
{\it Small loop transfer spaces with respect to subgroups of fundamental groups}, Topology Appl. \textbf {232}, 242-255, 2017.

\bibitem{PasGha}
H. Passandideh, F.H. Ghane and Z. Hamed, {\it On the homotopy groups of separable metric spaces}, Topology Appl.  \textbf{158}, 1607-1614, 2011.

\bibitem{Spa}
E.H. Spanier, {\em Algebraic Topology}, McGraw-Hill, New York, 1966.

\bibitem{Zim}
B. Zimmermann-Huisgen, {\it On Fuchs' problem 76},  J. Reine Angew. Math. \textbf{309}, 86-91, 1979.

\end{thebibliography}
\end{document}